\documentclass[acmtog,authorversion,nonacm,printcopyright]{acmart}
\usepackage{color, colortbl}
\usepackage{wrapfig}
\usepackage{multirow}
\usepackage{graphicx}
\definecolor{Gray}{gray}{0.9}
\usepackage{wrapfig,lipsum,booktabs}
\usepackage{units}
\usepackage{amsmath,amsfonts}
\usepackage{enumitem}
\usepackage{supertabular}
\usepackage{placeins}
\usepackage{float}
\usepackage{stfloats}
\setlist{noitemsep,leftmargin=*}
\newcommand{\R}[0]{\mathbb R}

\newcommand{\M}[0]{\mathcal M}
\newcommand{\eqdef}[0]{:=}

\DeclareMathOperator*{\argmin}{arg\min}

\newtheorem*{remark}{Remark}
\AtBeginDocument{%
  }

\setcopyright{cc}
\acmDOI{10.1145/3666087}

\acmJournal{TOG}

\begin{document}

\title{A Framework for Solving Parabolic Partial Differential Equations on Discrete Domains}

\author{Leticia Mattos Da Silva}
\orcid{0000-0001-7288-3015}
\affiliation{%
  \institution{Massachusetts Institute of Technology}
  \streetaddress{32 Vassar St}
  \city{Cambridge}
  \state{Massachusetts}
  \country{USA}
  \postcode{02139}
}
\email{leticiam@mit.edu}

\author{Oded Stein}
\affiliation{%
  \institution{University of Southern California}
  \streetaddress{941 Bloom Walk}
  \city{Los Angeles}
  \postcode{90089}
  \country{USA}}
\email{ostein@usc.edu}

\author{Justin Solomon}
\affiliation{%
  \institution{Massachusetts Institute of Technology}
  \streetaddress{32 Vassar St}
  \city{Cambridge}
  \country{USA}
  \postcode{02139}
}
\email{jsolomon@mit.edu}

\renewcommand{\shortauthors}{Mattos Da Silva et al.}

\begin{abstract}
  We introduce a framework for solving a class of parabolic partial differential equations on triangle mesh surfaces, including the Hamilton-Jacobi equation and the Fokker-Planck equation. PDE in this class often have nonlinear or stiff terms that cannot be resolved with standard methods on curved triangle meshes. To address this challenge, we leverage a splitting integrator combined with a convex optimization step to solve these PDE. Our machinery can be used to compute entropic approximation of optimal transport distances on geometric domains, overcoming the numerical limitations of the state-of-the-art method. In addition, we demonstrate the versatility of our method on a number of linear and nonlinear PDE that appear in diffusion and front propagation tasks in geometry processing.
\end{abstract}

\ccsdesc[500]{Mathematics of computing~Partial differential equations}
\ccsdesc[500]{Computing methodologies~Shape analysis}
\keywords{diffusion, optimal transportation, Hamilton-Jacobi, Fokker-Planck}
\begin{teaserfigure}
  \includegraphics[width=\textwidth]{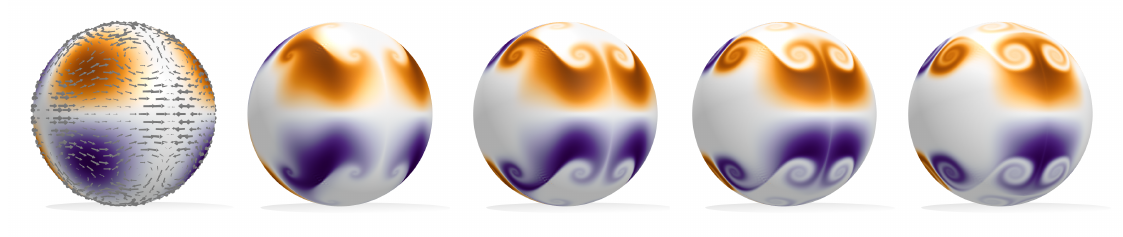}
  \caption{We show the time evolution of the Fokker-Planck equation under a certain flow on the discrete sphere with $81{,}920$ triangle faces and $40{,}962$ vertices, obtained using our framework. Our method can be used for any equation in a class of second-order parabolic PDE on triangle mesh surfaces.}
  \label{fig:teaser}
\end{teaserfigure}


\maketitle

\section{Introduction}
The analysis of partial differential equations (PDE) is a ubiquitous technique in computer graphics, geometry processing, and adjacent fields. In particular, parabolic PDE describe a wide variety of phenomena. For example, instances of the Hamilton-Jacobi equation model the time evolution of front propagation and the evolution of functions undergoing nonlinear diffusion. As another example, the Fokker-Planck equation describes the evolution of density functions driven by stochastic processes. Each of these equations has a long history as a means to study various problems in computer graphics, including modeling flames and fire \cite{nguyen2002fire}, stochastic heat kernel estimation \cite{armstrong2017stochastic}, medial axis detection \cite{du2004axis}, and texture synthesis \cite{witkin1991textures}. Hence, methods to solve this class of PDE over geometric domains are central in geometry processing.

Myriad numerical algorithms have been proposed for solving PDE in geometry processing. Unfortunately, the most popular algorithms are unsuitable for important regimes, such as capturing infinitesimal or nonlinear phenomena. An interesting example involves the convolutional Wasserstein distance method for barycenter computation \cite{solomon2015convolutional}. The aforementioned method is built on tiny amounts of diffusion, but relies on heuristics for choosing diffusion times. If the diffusion time step is too small, then the method leads to numerical inaccuracies, and if the step is too large, then the method results in approximations that quantitatively appear wrong. Another example involves nonlinear PDE used to describe certain types of flows. In this case, the challenge is that explicit integrators used for these PDE require time step restrictions to avoid numerical instability, and implicit integration schemes are not equivalent to solving a single linear system of equations, turning out to be too expensive.

To address these challenges, we propose a framework leveraging a splitting integration strategy and an appropriate spatial discretization to solve parabolic PDE over discrete geometric domains. The splitting allows us to leverage the implicit integration of a well-known PDE, the heat equation, and use a convex relaxation to deal with the challenging piece of the parabolic PDE. Empirically, our method overcomes limitations presented by state-of-the-art methods in geometry processing. We apply our framework to difficult parabolic PDE: log-domain heat diffusion, the nonlinear $G$-equation, and the Fokker-Planck equation, all of which we solve efficiently on a variety of domains and time steps.


\paragraph*{Contributions.} Our main contributions are: 
\begin{itemize}
    \item A numerical framework to solve linear and nonlinear parabolic PDE on curved triangle meshes with efficiency matching that of conventional methods in geometry processing.
    \item A log-domain diffusion algorithm that overcomes known limitations of geometry processing methods relying on tiny amounts of diffusion,  demonstrated on optimal transport tasks.
    \item An application of our framework to  numerical integration of the $G$-equation, which can be used as a component in graphics pipelines for the simulation of fire and flames.
\end{itemize}

\section{Related Work}

\subsection{PDE-driven geometry processing} 

PDE are a fundamental component of many geometry processing algorithms. PDE-based approaches have been used for surface fairing \cite{desbrun1999fairing,bobenko2005willmore}, surface reconstruction \cite{duan2004reconstruction,xu2006modelling}, and physics-based simulation \cite{chen2013modeling,obrien2002fracture}, to cite a few examples. Hence,  there has been a considerable amount of work focused on developing accurate and robust PDE solvers. The most popular PDE-based methods, however, such as the ``heat method" for geodesic distance approximations \cite{crane2013geodesics}, are concerned with solving linear PDE problems. 

\subsection{Second-order Parabolic PDE}

Second-order parabolic PDE comprise a general extension of the heat equation. The solutions to PDE in this class are in some ways related to the solutions of the heat equation; these equations appear in similiar applications where the object of interest is the density of a quantity on a domain, evolving forward in time. While the heat equation is a well-studied PDE with a wide array of tools available to solve it on discrete domains, second-order parabolic PDE in general are more challenging, in particular in the presence of additional terms that add nonlinearity or chaotic behaviour.

The typical strategy for spatial discretization of second-order parabolic PDE are Garlekin methods \cite{evans2010pde}. Common numerical integration techniques are finite difference schemes, including forward Euler, explicit Runge-Kutta formulas, backward Euler, and the Crank-Nicolson method. Still, many PDE in this family cannot be efficiently solved on triangle meshes with these discretization methods. Some of the difficulties that arise include the following: the strong stiffness of certain equations in this family is not suitable for methods such as Runge-Kutta \cite{sommeijer1998explicit}; and equations such as the Hamilton-Jacobi do not have a conservative form, making it difficult to write out fluxes for discontinuous Garlekin methods adapted to irregular domains \cite{yan2011garlekin}.

\subsubsection{Heat diffusion and entropy-regularized optimal transport} The entropy-regularized approximation of transport distances provides a method for solving various optimal transportation problems by realizing them as optimization problems in the space of probability measures equipped with the Kullback-Leibler (KL) divergence. We refer the reader to \cite{villani2003transportation} for a complete introduction to the optimal transportation problem, which roughly seeks to transport all the mass from a source distribution to a target distribution with minimal cost.  \cite{cuturi2013sinkhorn,benamou2014iterative} link entropy-regularized transport to minimizing a KL divergence.

\citet{cuturi2013sinkhorn} proposed a fast computational method to compute transport distances using entropic regularization based on the Sinkhorn algorithm. \citet{solomon2015convolutional} extended this work by showing that the kernel in the Sinkhorn algorithm for the 2-Wasserstein distance can be approximated with the heat kernel, which can be computed using PDE solution techniques over geometric domains. \citet{huguet2023geodesic} presents a similar attempt at improving Sinkhorn-based methods on discrete manifolds. This perspective is also echoed in Schr\"odinger bridge formulations of transport, whose static form includes an identical substitution of the heat kernel \cite{leonard2014dynamical}.

We leverage this approximation with the heat kernel over geometric domains, and show in \S\ref{subsec:barycenter}--\ref{subsec:measure-interp} that the numerical challenges arising for small entropy coefficients can be overcome by using our framework to solve a certain second-order parabolic nonlinear PDE.
\begin{figure}[t]
  \centering
  \includegraphics[width=\linewidth]{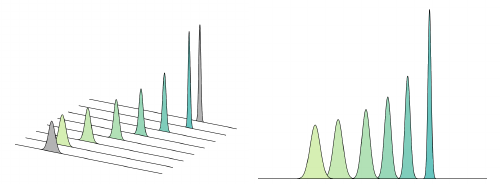}
  \caption{Wasserstein barycenters (green) between two distributions on the unit line (grey) with $1{,}000$ elements for varying pairs of weights. Our method obtains these barycenters for a tiny amount of entropy, i.e., $\gamma=10^{-7}$, whereas the method in  \citet{solomon2015convolutional} fails. \vspace{-0.2in}}
  \label{fig:line-barycenter}
\end{figure}

\subsubsection{$G$-equation}\label{subsubsec:Gequation} The $G$-equation is a level-set Hamilton-Jacobi equation introduced by F.A.\ Williams in \cite{buckmaster1985combustion} to model turbulent-flame propagation in combustion theory. While some finite difference schemes have succeeded in solving various cases of the $G$-equation on regular grids \cite{liu2013turbulent,gu2021incompressible}. Several challenges arise in the curved triangle mesh setting: for certain values of flow intensity, the constraint on time step---given by the Courant–Friedrichs–Lewy (CFL) condition---is quite restrictive on explicit methods; and the nonlinearity of the equation yields an open problem for implicit methods. 

While the $G$-equation is well-known in computational fluid dynamics (CFD) \cite{nielsen2022combustion}, it was brought to the computer graphics community as the ``thin-flame model" to simulate flame and fire in \cite{nguyen2002fire}. This work was later extended to model a wider variety of combustion phenomena in graphics, e.g., see \cite{hong2007wrinkled}. The time integration implemented in these graphics pipelines are based on the foundational scheme introduced by \citet{osher1988curvature}, which deals with the aforementioned limitations, in particular, time step restrictions due to a CFL condition.

As we will show in \S\ref{subsec:numerical-fire}, our framework achieves better numerical stability than this standard scheme for regimes in violation of its CFL condition, and matches numerical results for regimes within its CFL condition, where the standard scheme provides a reasonable approximation of the exact solution.

\subsubsection{Fokker--Planck equation}\label{subsubsec:FPequation}  The Fokker-Planck equation is a linear parabolic partial differential equation describing the time evolution of the probability density function of a process driven by a stochastic differential equation (SDE). Processes driven by SDE and the Fokker-Planck equation appear in applications in computer vision and image processing \cite{smolka1997images, leatham2022xray}. In the context of computer graphics and geometry processing, the Fokker-Planck equation and its variants have been used for texture synthesis based on nonlinear interactions \cite{witkin1991textures} and stochastic kernel estimation \cite{armstrong2017stochastic}.

We show that our framework can be used to solve the Fokker-Planck equation directly on curved triangle meshes, in contrast to using its SDE formulation. In this way, our work lays the foundation for a new approach to using stochastic differential equations in traditional geometry processing.

\begin{figure}[t]
  \centering
  \includegraphics[width=\linewidth]{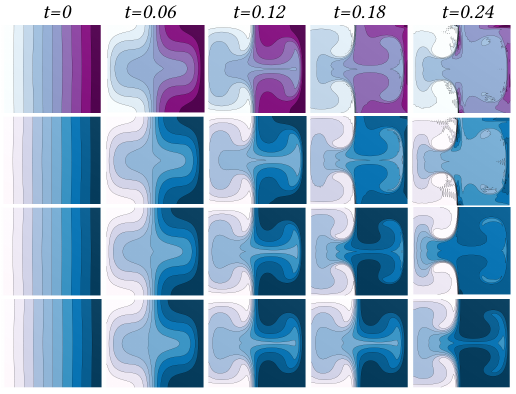}
  \caption{Level-sets of the $G$-equation under cellular flow obtained via our method with varying values of viscosity: $\varepsilon=0, 10^{-4},10^{-3},$ and $10^{-2}$ (from top to bottom, respectively). See Figure \ref{fig:types-flow} for an illustration of cellular flow. \vspace{-0.6cm}}
  \label{fig:levelsets}
\end{figure}

\subsection{Convex Optimization for Regularized Geodesics}

Perhaps the closest related work to our paper is ~\cite{edelstein2023convex}, which proposes a convex optimization technique for extracting regularized geodesic distances on triangle meshes. Their paper generalizes the PDE-based results in \cite{belyaev2020distance, belyaev2015variational} from estimation of distance-like functions to general regularizers under some mild conditions. This optimization approach is based on the equivalence found between the solution of an eikonal type equation, i.e., a \textit{stationary} first-order PDE dependent only on the gradient of a variable function, and the solution to a volume maximization problem. The equivalence is a consequence of mathematical results on the study of maximal viscosity subsolutions of eikonal type equations first introduced by \cite{ishii1987hamiltonjacobi}. In our case, we handle the problem of more general \textit{evolutionary} Hamilton-Jacobi equations with Hamiltonians depending on the variable function rather than just its gradient.

\section{Mathematical Preliminaries}

\subsection{General Preliminaries}  \label{subsec:preliminaries}

Let $\M \subset\R^3$ be a  {compact} surface embedded in $\R^3$, possibly with boundary $\partial \M$; we will use $\hat n(x)$ to denote the unit normal to the boundary $\partial\M$ at $x\in\partial\M$.  Take $\nabla$ to be the gradient operator and $\Delta$ to be the Laplacian operator, with the convention that $\Delta$ is negative semidefinite.  We refer the reader to \cite{jost2011laplace} for a general introduction to Laplacian operators, which roughly map functions to their total second derivative pointwise.

Let $q$ denote $\nabla u(x)$. In this paper we consider \textit{evolutionary}, that is, time-varying, second-order parabolic PDE of the following form:
\begin{equation}\label{eq:pdeform}
    \frac{\partial u}{\partial t} + H(x,q,u) = \varepsilon \Delta u,
\end{equation}
where $\varepsilon\geq0$ and $u(x,t)\colon\M \times [0,\infty)\to\R$ is an unknown variable function. If $\varepsilon>0$, then $u$ is guaranteed to be twice differentiable on $\M$. The function $H(x,q,u)\colon T^*\M \times \R\to\R$ is the Hamiltonian on $T^*\M$ of 
\begin{equation}\label{eq:HJ}
    \frac{\partial u}{\partial t} + H(x,q,u) = 0,
\end{equation}
which is known as the Hamilton-Jacobi equation \cite{evans2010pde}. The parameter $\varepsilon$ in \eqref{eq:pdeform} is called the viscosity parameter and, as $\varepsilon\to0$, the function $u$ converges to a solution of \eqref{eq:HJ} \cite{fleming1989games}. 

We tackle the Cauchy problem for \eqref{eq:pdeform}, that is, with a prescribed initial condition $u(x,0)=u_0(x)$, where $u_0\colon\M\to\R$ is some scalar function, and boundary conditions determining the behavior of $u(x,t)$ at $x\in\partial\M$. Two common choices for boundary conditions are \emph{Dirichlet} conditions, which prescribe the values of $u(x,t)$ for all $(x,t)\in\partial\M\times(0,\infty)$, and \emph{Neumann} conditions, which specify $\nabla u(x,t)\cdot\hat n(x)\equiv0$ for all $x\in\partial\M$.

\subsection{Example PDE}

Our framework arose in our study of functions $u(x,t)$ satisfying the following PDE:
\begin{equation}\label{eq:logheat}
    \frac{\partial u}{\partial t} - \|\nabla u\|_2^2 = \Delta u.
\end{equation}

This PDE will be referred to in this paper as the ``nonlinear diffusion" equation (see \S\ref{subsec:nonlineardiffusion}). Its name originates from the nonlinear nature of the PDE and its relationship with heat diffusion shown in the proposition below.
\begin{proposition} \label{prop-heat}
Suppose $v(x,t)$ is such that $v(x,t)>0$ for all $(x,t)\in\M\times[0,\infty)$ and $v$ satisfies the heat equation:
\begin{equation}\label{eq:heat}
    \frac{\partial v}{\partial t} = \Delta v.
\end{equation}
Define $u(x,t)\eqdef\log v(x,t)$.  Then, $u(x,t)$ satisfies (\ref{eq:logheat}). Moreover, if $v$ satisfies Dirichlet or Neumann boundary conditions, then $u$ satisfies the same boundary conditions up to applying $\log$. In sum, $u$ satisfies the boundary conditions satisfied by $v$ up to applying $\log$.
\end{proposition} 
\begin{proof} Let $u=\log v$, then this is a straightforward application of the chain rule:
\begin{equation*}
\begin{array}{rcl}
    \frac{\partial u}{\partial t} - \|\nabla u\|_2^2 &=& \Delta u\\[5pt]
    \implies \frac{1}{v} \frac{\partial v}{\partial t} - \frac{1}{v^2}\|\nabla v \|_2^2 &=& \nabla \cdot \big(\frac{1}{v} \nabla v \big)\\[5pt]
    &=& -\frac{1}{v^2} (\nabla v \cdot \nabla v) + \frac{1}{v} \Delta v\\[5pt]
    \implies\frac{\partial v}{\partial t}&=&  \Delta v 
    \end{array}
\end{equation*}
\end{proof}

\begin{remark}
Proposition \ref{prop-heat} only applies to strictly positive functions $v(x,t)$; by the maximum principle, it is sufficient to check $v(x,0)>0$ to guarantee this condition for all $t\in[0,\infty)$. Also, $u$ is a supersolution to the heat equation since clearly $\nicefrac{\partial u}{\partial t}\geq\Delta u$.
\end{remark}

The overarching theme of this paper is that the same numerical procedure that we developed to solve \eqref{eq:logheat} can be generalized to handle other parabolic problems of the form \eqref{eq:pdeform} given certain continuity and convexity assumptions on $H$ and $u$. In particular, we assume the following:

\begin{itemize}[leftmargin=20pt]
    \item[(A1)] \emph{Continuity.} The function $H(x,q,u)$ is continuous on $T^*\M\times \R$.
    \item[(A2)] \emph{Monotone convexity.} $\exists c\in \R$ such that, whenever $w\leq u$, we have $H(x,q,u)-H(x,q,w)\geq c(u-w)$.
    \item[(A3)] \emph{Lipschitz.} Locally on $\M$, we have $\|H(x,q,u)-H(y,q,u)\|_2\leq L(1+\|q\|_2)\|x-y\|_2$ , where $L$ is a Lipschitz constant.
\end{itemize}

In this paper, we consider three choices of the function $H=H(x,q,u)$: 
\subsubsection{Nonlinear diffusion} \label{subsec:nonlineardiffusion} First, motivated by equation \eqref{eq:logheat}, we consider $H(x,q,u)=-\|q\|_2^2$. In \S\ref{subsec:barycenter}--\ref{subsec:measure-interp}, we will show that the nonlinear diffusion equation together with Proposition \ref{prop-heat} play a key role in overcoming the limitations faced by methods relying on small amounts of diffusion to compute entropy regularized transport distances on discrete domains.

Here, we verify assumptions (A1)--(A3) are fufilled: (A1) follows from the continuity of the $L_2$ norm; (A2) and (A3) are empty conditions in this case since $H$ does not depend on $x$ or $u$. These assumptions can be verified straightforwardly for remaining example PDE by the interested reader.

\subsubsection{G-equation} \label{subsec:Gequation} Second, let $\Phi\colon\M\to\R^3$ be a vector field, we consider $H(x,q,u)=\Phi(x)\cdot q - \|q\|_2$, which corresponds to the $G$-equation (see \S\ref{subsubsec:Gequation}):
\begin{equation}\label{eq:G-equation}
    \frac{\partial u}{\partial t} + \Phi(x)\cdot \nabla u - \|\nabla u\|_2 = \varepsilon\Delta u.
\end{equation}

For $\varepsilon>0$, equation \eqref{eq:G-equation} is known as the \textit{viscous} $G$-equation, and when $\varepsilon=0$, it is known as the \textit{inviscid} $G$-equation. 

\subsubsection{Fokker-Planck equation} \label{subsec:fokker-planck} Finally, we consider $H(x,q,u)=u\nabla\cdot\Phi(x) + q\cdot\Phi(x)$, which corresponds to the Fokker-Planck equation (see \S\ref{subsubsec:FPequation}): 

\begin{equation}\label{eq:fokker-planck}
    \frac{\partial u}{\partial t} + u\nabla\cdot\Phi(x) + \nabla u\cdot\Phi(x) = \varepsilon\Delta u.
\end{equation}
Here $u$ is the density function of the trajectories of the following SDE:
\begin{equation}\label{eq:sde}
    \mathrm{d}x = \nabla\cdot\Phi(x)\mathrm{d}t + \sqrt{2\varepsilon}\;\mathrm{d}W(t),
\end{equation}
where $W(t)$ is a Wiener process. We refer the reader to \cite{medved2020langevin} for a review on the relationship between equations \eqref{eq:fokker-planck}--\eqref{eq:sde}. The vector field $\Phi$ is typically known as the drift vector and $\varepsilon$ as the diffusion coefficient, but we will call the latter the viscosity parameter for consistency throughout this paper.

\begin{figure}
  \centering
  \includegraphics[width=\linewidth]{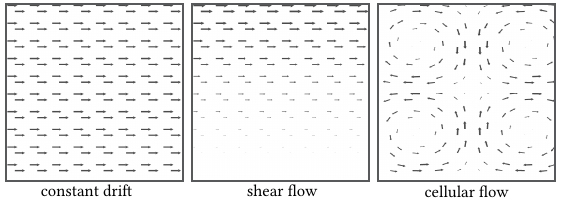}
  \caption{Three typical examples of vector fields $\Phi$ used when evolving second-order parabolic PDE that involve terms with vector fields.}
  \label{fig:types-flow}
\end{figure}

\subsection{Viscosity Solutions}

We introduce a few definitions from \cite{crandall1988viscosity} in the study of weak solutions to Hamilton-Jacobi equations. These definitions will be necessary in the proofs presented in section \S\ref{subsubsec:hterm}. We refer the reader to \cite{crandall1992users} for a complete introduction to the theory of viscosity solutions and their applications to PDE.

\begin{definition}[Viscosity subsolutions, and resp., supersolutions]
     Let $V$ be an open subset in the manifold $\M$. A function $u\colon V \to \R$ is a viscosity subsolution (resp., supersolution) of $\frac{\partial u}{\partial t}+H(x,q,u)=0$ if for every $C^1$ function $\varphi\colon V\to \R$ and every point $x\in V$ such that $u-\varphi$ has a local maximum (resp., minimum) at $x$, we have $\frac{\partial \varphi}{\partial t}+H(x,\nabla \varphi,\varphi)\leq0$ (resp., $\frac{\partial \varphi}{\partial t}+H(x,\nabla \varphi,\varphi)\geq0$).
\end{definition}

\begin{figure}[t]
  \centering
  \includegraphics[width=\linewidth]{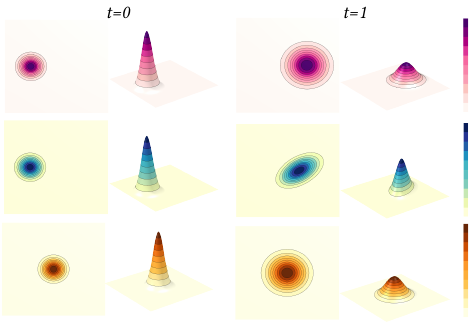}
  \caption{Time evolution of the Fokker-Planck equation \eqref{eq:fokker-planck} on a $100\times100$ triangle grid, obtained using our method, under constant drift (top), shear flow (middle), and no drift (bottom). See Figure \ref{fig:types-flow} for an illustration of constant drift and shear flow.}
  \label{fig:grid-fokker-planck}
\end{figure}

\begin{definition}[Viscosity solution]
     A function $u\colon V\to \R$ is a viscosity solution of $\frac{\partial u}{\partial t}+H(x,q,u)=0$ if it is both a subsolution and a supersolution.
\end{definition}

We note that the above definitions do not assume the function $u$ to be continuous. In fact, the necessity for such definitions arose in the study of problems without continuous solutions. As stated in \S\ref{subsec:preliminaries}, the functions we study are twice differentiable if $\varepsilon>0$ and hence the following theorem is useful:

\begin{theorem}[Continuously differentiable viscosity solution]
    A $C^1$ function $u\colon V\to\R$ is a viscosity solution if and only if it is a classical solution.
\end{theorem}

Hence, while we might still use the term viscosity solution throu- ghout the paper, when referring to the problem in \eqref{eq:HJ} with $\varepsilon>0$, a viscosity solution is equivalent to a solution in the classical sense.

\begin{figure}
  \centering
  \includegraphics[width=0.9\linewidth]{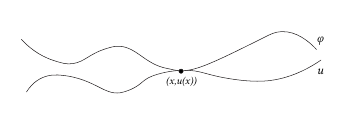}
  \caption{A two dimensional illustration of a viscosity supersolution $\varphi$. \vspace{-0.2in}}
  \label{fig:viscosity-supersolutions}
\end{figure}

\section{Method}

Let $\Omega=(V,E,F)$ be a triangle mesh with vertices $V\subset\R^3$, edges $E\subseteq V\times V$, and triangular faces $F\subset V\times V\times V$. In this section, we develop a general approach to approximating solutions to equations of the form \eqref{eq:pdeform} on triangle meshes, given $u(x,0)$ discretized on $(V,F)$ using one value per vertex, $u_0\in\R^{|V|}$. 

In geometry processing, the most common discretization of the PDE on triangle meshes uses the \emph{finite element method} (FEM). Following \cite{botsh2010polygon}, take $L\in\R^{|V|\times|V|}$ to be the cotangent Laplacian matrix associated to our mesh (with appropriate boundary conditions), and take $M\in\R^{|V|\times|V|}$ to be the diagonal matrix of Voronoi cell areas. Also, take $G\in\R^{3|F|\times|V|}$ to be the matrix mapping vertex scalar values to per face gradients. In our time discretization strategy, we will denote the time step size by $h$.

In what follows, we first describe our method for time integration, and then we define a spatio-dicretization strategy suitable to the different terms associated with PDE of the form in \eqref{eq:pdeform}.

\subsection{Time Integration}



Two typical means of solving PDE in time are explicit time integration, such as the \textit{forward Euler} method, and implicit integration, such as the \textit{backward Euler} method. 

\subsubsection{Forward Euler} Forward Euler often imposes a strict restriction (upper bound) on $h$ for stability.  Runge--Kutta and other variations of this integrator can improve stability and accuracy of forward Euler integration, but almost all have two critical drawbacks:

\begin{itemize}
    \item Time step restrictions require that we take many small steps ($h\ll1$) to avoid introducing instability. 
    \item If initial conditions are nonzero only on one vertex $v\in V$, as might be the case for algorithms based on taking the logarithm of the heat kernel (see \S\ref{subsec:barycenter}--\ref{subsec:measure-interp}), it takes $k\sim O(|V|)$ steps of this integrator---regardless of $h$---before the solution is nonzero everywhere (which is needed to apply the logarithm).
\end{itemize}

\subsubsection{Backward Euler} Backward Euler is unconditionally stable and diffuses information everywhere along the domain in a single step, addressing the two issues of forward Euler. It is first-order accurate in $h$, which often suffices for small $h>0$. If the PDE contains a nonlinear term, however,  backward Euler and related implicit integration schemes are no longer equivalent to solving a single linear system of equations; in this case, implicit integration leads to a nonlinear root-finding problem that requires a nearby guess.

We now outline a method to address the challenges faced by classical approaches in the context of solving second-order parabolic PDE.

\subsubsection{Strang splitting}\label{sec:strang} We would like to leverage the effectiveness of implicit integration for the simplest version of \eqref{eq:pdeform} where $H=0$, but without having to solve a nonlinear system of equations. To this end, we propose using \emph{Strang--Marchuk splitting} (also known \begin{wrapfigure}{l}{0.2\textwidth}
    \vspace{-.2in}
    \includegraphics[width=0.24\textwidth]{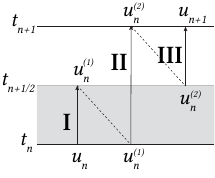} \hspace{-.1in} \vspace{-.25in}
\end{wrapfigure} as leapfrog or St\"ormer--Verlet integration), originally proposed in \cite{strang1964split, marchuk1988split}. Generically, operator splitting methods solve differential equations $\nicefrac{dx}{dt}=f(x)+g(x)$ by alternating steps in which the terms $f$ and $g$ on the right-hand side are treated individually. This splitting scheme is justified by the Lie--Trotter--Kato formula \cite{trotter1959product,kato1974trotter}. In our case, we choose a splitting so that one piece of the splitting resembles solving the  heat equation and the other benefits from the solution techniques along the lines of those introduced by \citet{belyaev2020distance} (later extended by \citet{edelstein2023convex}).

In particular, suppose $u_{n-1}\in\R^{|V|}$ is our estimate of $u$ at time $h(n-1)$.  We estimate $u_n$ via the following three steps:
\begin{enumerate}[label=\Roman*.]
    \item First, apply half a time step of implicit heat diffusion, approximating the solution to $\nicefrac{\partial u}{\partial t}=\varepsilon\Delta u$ at $t=h(n-1)+\nicefrac h2$:
    \begin{equation}\label{eq:halfheat1}
        u_n^{(1)} = \left( M - \left(\nicefrac{h}{2}\right)\varepsilon L \right)^{-1} M u_{n-1}.
    \end{equation}
    \item Next, apply a full time step approximating the solution to $\nicefrac{\partial u}{\partial t} + H(x,q, u) = 0$, as detailed in \S\ref{subsubsec:hterm}, to obtain $u_n^{(2)}$.
    \item Finally, apply a second half-step of implicit heat diffusion:
    \begin{equation}\label{eq:halfheat2}
        u_n = \left( M - \left(\nicefrac{h}{2}\right)\varepsilon L \right)^{-1} M u_n^{(2)}.
    \end{equation}
\end{enumerate}  
Note that \eqref{eq:halfheat1} and \eqref{eq:halfheat2} solve the same linear system with different right-hand sides, allowing us to pre-factor the matrix $M-\left(\nicefrac{h}{2}\right)\varepsilon L$ if we plan to apply our operator more than once.  It is also worth noting that possible time step size limitations in Strang splitting are only imposed by the method chosen to solve each subproblem. The backward Euler substep is unconditionally stable and the convex optimization substep, which is outlined next, enjoys similar stability to implicit methods. Analytically, Strang splitting has convergence of second order \cite{strang1964split}. \hfill\break\indent

There are different splitting schemes that one could choose when treating different kinds of PDE. We considered accuracy in choosing an suitable splitting scheme for our case. The Lie-Trotter scheme, for instance, is first-order and can provide crude results for nonlinear problems \cite{strang1964split}. Strang splitting is second-order, i.e., $\mathcal{O}(h^2)$ given a time step $h$. The nonlinearity in various second-order parabolic PDE makes Strang a preferable choice.

\subsubsection{Strang-Marchuk splitting, step II} \label{subsubsec:hterm} We use convex optimization to take a time step of the first-order equation $\nicefrac{\partial u}{\partial t} + H(x,q, u)=0$, including the case where this is a nonlinear first-order problem. In this section, we describe an implicit integrator leveraging the fact that $-H(x,q,u)$ is jointly convex in its arguments $(q,u)$.

In this step, our goal is to obtain $u_n^{(2)}$ by integrating the first-order equation for time $h$ starting at function $u_n^{(1)}$.  The simplest implicit integration strategy is \emph{backward Euler} integration, which would solve the following root-finding problem:
\begin{equation}\label{eq:implicit}
\frac{u_n^{(2)}-u_n^{(1)}}{h}+H(x,\nabla u_n^{(2)},u_n^{(2)})=0.
\end{equation}
Here, we will assume the $u_n^{(i)}$ are functions on the surface $\M$, but an identical formulation will apply after spatial discretization in \S\ref{sec:spatialdiscretization}.  This problem is nonlinear whenever $H$ is nonlinear, making it difficult to solve.

Instead, we observe that relaxing the equality in \eqref{eq:implicit} to an inequality leads to a \emph{convex} constraint on the unknown function $u_n^{(2)}$.  Taking inspiration from \citet{edelstein2023convex}, we arrive at the following optimization problem to advance forward in time:
\begin{equation}\label{eq:optprob}
    u_n^{(2)}=
    \left\{
    \begin{array}{rl}
    \arg\min_u \quad&\int_\M u(x) \,\mathrm{dVol}(x)\\
    \mathrm{subject \; to} \quad & 
    \frac1h(u-u_n^{(1)}) + H(x,\nabla u, u) \geq 0\\
    &\qquad\mathrm{\; for \; all\;} x\in\M.
    \end{array}
    \right.
\end{equation}
Notice that the constraint in \eqref{eq:optprob} yields a \emph{supersolution} of the implicit problem \eqref{eq:implicit}.  Moreover, formulation \eqref{eq:optprob} is a convex problem for $u$ since it has a linear objective and a pointwise convex constraint. 

The intuitive argument for the equivalence between solutions to eikonal type equations and solutions to a convex optimization problem given in \cite{belyaev2020distance}, and later used by \cite{edelstein2023convex} is not sufficient to obtain the well-posedness of \eqref{eq:optprob} and its equivalence to \eqref{eq:implicit}. These works are similar to ours in that they establish an equivalence between a first-order parabolic PDE and a constrained optimization problem by leveraging results from the theory of viscosity solutions. A key distinction, however, lies in the fact that PDE of the form \eqref{eq:implicit} are \textit{evolutionary} and could also have dependence on the value of the unknown function $u$ rather than just its gradient $\nabla u$. In what directly follows, however, we justify using \eqref{eq:optprob} to solve \eqref{eq:implicit}.

\begin{theorem}[Comparison principle]\label{thm:comparison}
    Assume $H$ satisfies assumptions (A1)--(A3). Let $w$ be a viscosity subsolution and $u$ be a viscosity supersolution of $\nicefrac{\partial u}{\partial t} + H(x,q,u) = 0$. Assume that $w,u$ are locally bounded and $w(x,0)\leq u(x,0)$ for all $x$. Then $w(x,t)\leq u(x,t)$ for all $(x,t)$.
\end{theorem}

\begin{theorem}\label{thm:uniqueness}
    Assume $H$ fulfills (A1)--(A3). Then \eqref{eq:implicit} has a unique solution in the viscosity sense.
\end{theorem}

The existence and uniqueness results for Cauchy problems of the form $\nicefrac{\partial u}{\partial t} + H(x,q,u) = 0$ on an open subset of $\R^N$, under essentially the same assumptions (A1)--(A3), is given in \cite{barles2013viscosity} (see \S5 Theorem $5.2$). The comparison principle proof for the same Cauchy problem is also given by \citet{barles2013viscosity} (see \S7 Theorem $7.1$). Readers can also refer to Appendix \ref{sec:comparison} and \ref{sec:uniqueness} for a sketch of these proofs. While the domain in \cite{barles2013viscosity} is an open subset of $\R^N$, the authors in \cite{peng2008maximum} justify the extension of Theorems \ref{thm:comparison}--\ref{thm:uniqueness} to manifolds with or without a boundary.

\begin{corollary}
    Under the same assumptions on $H$, the minimal viscosity supersolution to \eqref{eq:implicit} is the equation's unique solution in the viscosity sense.
\end{corollary}

With these results in hand, we can discuss the equivalence between the solution of that problem and the unique solution of \eqref{eq:implicit}.

\begin{theorem}
    Again, assume $H$ fullfils (A1)--(A3). The unique viscosity solution to \eqref{eq:implicit} is the solution to the minimization problem in \eqref{eq:optprob}.
\end{theorem}
\begin{proof}
    The only variable functions that can satisfy the inequality constraint in \eqref{eq:optprob} are by definition viscosity supersolutions of \eqref{eq:implicit}. The minimal viscosity supersolution is the variable function that minimizes the integral volume $\int_\M u(x)\mathrm{dVol}(x)$. The uniqueness given by Theorem \ref{thm:uniqueness} and the comparison principle together guarantee that the minimal viscosity supersolution is the viscosity solution to \eqref{eq:implicit}. Hence, the minimal supersolution not only satisfies the inequality constraint but actually achieves equality, meaning the constraint in our convex relaxation becomes tight. 
\end{proof}

So far, in this subsection, we have shown the existence and uniqueness of a solution to \eqref{eq:implicit}, as well as a comparison principle for viscosity solutions of this equation. We used these results to obtain the equivalence of the solution of \eqref{eq:implicit} to the solution of the convex optimization problem outlined in \eqref{eq:optprob}. Recall that this equivalence to a convex optimization problem is the key insight we use to solve one of the steps in our Strang splitting scheme; our main goal is to solve PDE of the form \eqref{eq:pdeform}.

In the spatially-discrete case (see \S\ref{sec:spatialdiscretization}), well-posedness of our convex program is immediate since it has a linear objective and satisfiable convex constraints.  Empirically, we find that our discretization of this time step closely resembles ground-truth when it is available.

\begin{figure*}
  \centering
  \includegraphics[width=\linewidth]{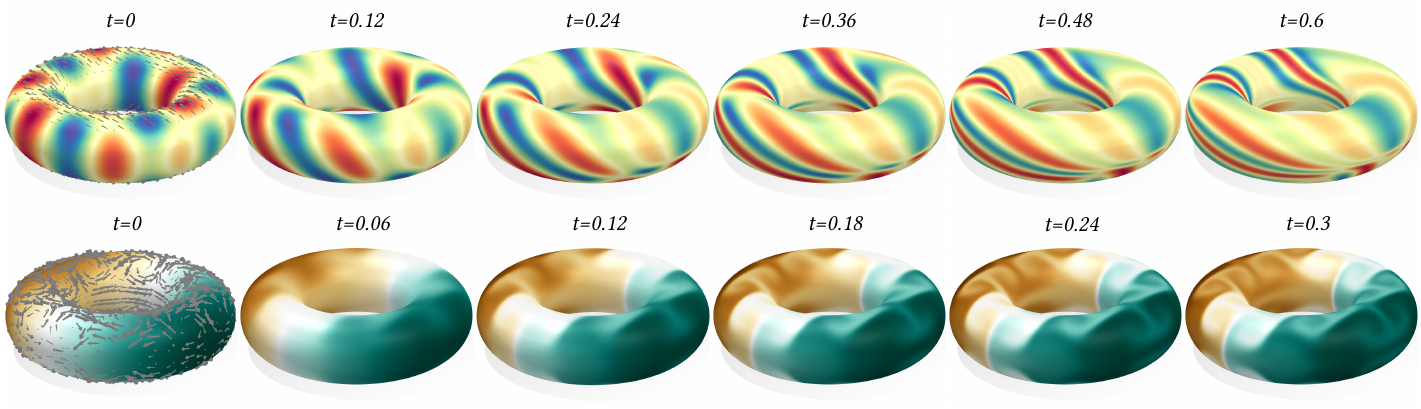}
  \caption{Results of our framework's implementation on the discrete tori with $200{,}000$ triangle faces and $100{,}000$ vertices. Top row: time evolution of the Fokker-Planck equation \eqref{eq:fokker-planck} under a certain flow. Bottom row: time evolution of the $G$-equation \eqref{eq:G-equation} under Kolmogorov flow.}
  \label{fig:tori-evolution}
\end{figure*}

\subsection{Spatial Discretization}\label{sec:spatialdiscretization}

We can approximate solutions to equation \eqref{eq:pdeform} by solving the ordinary differential equation:
\begin{equation}\label{eq:odeform}
    \frac{du}{dt} + H(x,Gu,u) = \varepsilon M^{-1} L u.
\end{equation}
We will apply the Strang splitting strategy from \S\ref{sec:strang} to this ODE, yielding steps nearly identical to those outlined in the previous section.

Splitting time integration with a step of implicit diffusion has been used in other corners of computer graphics---we refer the reader to \citet{elcott2007fluids} for an example of such an algorithm. Their implementation uses a spatial discretization via Discrete Exterior Calculus (DEC). A number of the PDE that we consider in our framework, e.g., the G-equation, do not fall into the category for which DEC is a pertinent choice of spatial discretization. We define a different choice of discretization that is suitable to the PDE that we solve. The success of the aforementioned algorithms, however, suggests the viability of splitting time integration in geometry processing and computer graphics. Moreover, in cases where spatial discretization via DEC is available to a PDE in the class we consider, our splitting and convex relaxation strategy is likely relevant to time integration of the resulting ODE.

If we use a na\"ive piecewise-linear finite element method (FEM) to discretize \eqref{eq:pdeform} spatially, the $\nicefrac{\partial u}{\partial t}$ and $\Delta u$ terms are associated to the vertices of our mesh, while $H(x,Gu,u)$ includes terms that are associated to both the vertices ($x,u$) and the triangular faces ($Gu$). Hence, we propose an alternate spatio-discretization strategy to map any face valued quantities in $H$ to per-vertex values. In particular, on a mesh, we reformulate \eqref{eq:odeform} via
\begin{equation}\label{eq:vertex-pdeform}
    \frac{du_i}{dt} + \omega_i\sum_{j\sim i} a_j H_{ij}((Gu)_j(t),u_i(t)) = \varepsilon (M^{-1} L u(t))_i.
\end{equation}
where $\omega_i = \nicefrac{1}{\sum_{j\sim i} a_j}$ and $a_i$ is the area of the triangle $j$ in the one-ring neighborhood of vertex $i$. The right-hand side of equation  \eqref{eq:vertex-pdeform} is convex whenever $H_{ij}$ is jointly convex in its inputs. Hence, our discretizations below of \eqref{eq:HJ}, \eqref{eq:G-equation}, and \eqref{eq:fokker-planck} are convex since $H_{ij}$ is convex by design whenever $H$ is convex in the continuous case. 

Next, we define $H$ discretized as a per-vertex quantity for each of the three example PDE given in sections \S\ref{subsec:nonlineardiffusion}--\ref{subsec:fokker-planck}.

\subsubsection{Nonlinear diffusion}\label{sec:nonlineardiffusion} We can now discretize $H=-\| q\|_2^2$ on the right-hand side of \eqref{eq:logheat} as follows
\begin{equation}\label{eq:HJ-Hfun}
    H_i = -\omega_i \sum_{j\sim i}\sum_{k=1}^3 (G_ku\odot G_ku)_j
\end{equation}
where $\odot$ denotes element-wise multiplication, $G_k\in\R^{|F|\times|V|}$ is the matrix mapping a function on the vertices $V$ of our mesh to the $k$th components of its per-triangle gradient, and $W$ is the diagonal matrix whose elements are the weights $\omega_i$.
\subsubsection{G-equation} Similarly, we can discretize the $H$ function for \eqref{eq:G-equation} as follows 
\begin{equation}\label{eq:G-Hfun}
    H_i = \omega_i\sum_{j\sim i}\Big(\textstyle\sum_{k=1}^3  ( \Phi_k \odot G_ku)_j - \sqrt{\textstyle\sum_{k=1}^3 (G_ku)_j\odot (G_ku)_j}\;\Big),
\end{equation}
where $\Phi_k\in\R^{|F|}$ is the $k$th component of a vector field $\Phi(x)\in \R^{3|F|}$.

\subsubsection{Fokker-Planck equation}\label{sec:fokkerplanck} Finally, discretize the divergence operator $\nabla \cdot$ by $G^\mathrm{T}M_F\in\R^{|V|\times 3|F|}$, where $M_F$ is the diagonal matrix of triangle areas. The per-vertex discretization of the $H$ functional on the right-hand side of equation \eqref{eq:fokker-planck} becomes:
\begin{equation}\label{eq:FP-Hfun}
    H_i = (u\odot (G^\mathrm{T}M_F\Phi))_i+\omega_i \sum_{j\sim i}\sum_{k=1}^3 ( G_ku \odot \Phi_k )_j.
\end{equation}

\section{Implementation Details}

After introducing the spatial discretization techniques from \S\ref{sec:spatialdiscretization}, our nonlinear time step \eqref{eq:optprob} becomes a finite-dimensional convex optimization problem:
\begin{equation}
\begin{array}{r@{\ }l}
    \text{minimize}_{u_n^{(2)}}\; &\textstyle\sum_i {\big(u_n^{(2)}\big)}_i\\
    \text{subject to}\; & u_n^{(2)} - u_n^{(1)} + hH(x,Gu_n^{(2)},u_n^{(2)})  \geq 0,
    \end{array}
\end{equation}
where $H$ denotes the discretization of the smooth function fufilling assumptions (A1)--(A3) for each of our example PDE (see \S\ref{sec:nonlineardiffusion}-\S\ref{sec:fokkerplanck}).

Our implementation uses the CVX software library \cite{cvx,gb08} equipped with the default conic solver SDPT$3$ \cite{toh1999semidefinite,ttnc2003solving}. CVX is a free software that turns \textsc{Matlab} into a modeling language. Readers familiar with \textsc{Matlab} can use CVX to implement our framework with little effort since it allows them to write constraints and objectives using \textsc{Matlab}'s standard syntax. In principle, however, any convex solver (e.g., Mosek, Gurobi, ADMM) can be used.

To be concrete, we formulate the convex optimization problem for each of the example parabolic PDE of the form in \eqref{eq:pdeform} as follows:

\begin{table*}[!htbp]
    \begin{minipage}{\textwidth}
    \caption{\centering  Determination of order for convergence in Figs. \ref{fig:sphere-time-convergence},\ref{fig:bunny-time-convergence} and \ref{fig:fandisk-time-convergence} }
    \vspace{-5pt}
    \centering
    \begin{tabular}{cccccccc}
    \hline
    &Fig. \ref{fig:sphere-time-convergence}  & Fig. \ref{fig:sphere-time-convergence}   & Fig. \ref{fig:sphere-time-convergence} & &Fig. \ref{fig:bunny-time-convergence} & &Fig. \ref{fig:fandisk-time-convergence}\\
    & $m=-4$ &  $m=-3$  &$m=-2$& & & & \\
    \hline
     $h$&  R  &  R &  R  & $h$ &  R & $h$ & R\\
    \hline
    \rowcolor{Gray}
    $6.00\times10^{-1}$& $0.830$& $0.852$ & $ 0.968$ &$5.00\times10^{-2}$ & $0.650$ & $1.50\times10^{-3}$&$0.722$ \\
    $3.00\times10^{-1}$ & $0.921$& $0.935$& $0.989$ & $2.50\times10^{-2}$&$0.775$ & $7.50\times10^{-4}$& $0.801$ \\
    \rowcolor{Gray}
    $1.50\times10^{-1}$&$0.978$& $0.986$& $1.010$ &$1.25\times10^{-2}$ &$0.872$ & $3.75\times10^{-4}$ & $0.870$ \\
    $7.50\times10^{-2}$ & $1.025$& $1.029$& $ 1.040$ &$6.25\times10^{-3}$ &$0.936$ & $1.88\times10^{-4}$ & $0.938$\\
    \rowcolor{Gray}
    $3.75\times10^{-2}$ &$1.088$& $1.091$& $1.096$ & $3.13\times10^{-3}$& $1.018$& $9.38\times10^{-5}$ & $1.024$\\
    $1.88\times10^{-2}$ &$1.217$& $1.218$& $1.220$ &$1.56\times10^{-3}$ &$1.171$ & $4.69\times10^{-5}$ & $1.171$\\
    \rowcolor{Gray}
    $9.38\times10^{-3}$ &$1.582$& $1.583$& $1.584$ & $7.81\times10^{-4}$&$1.564$ & $2.34\times10^{-5}$ & $ 1.552$\\
    \hline
    \end{tabular}
    \vspace{.2in}
    \label{tab:time-convergence}
    \end{minipage}
\end{table*}

\begin{figure}
  \centering
  \includegraphics[width=\linewidth]{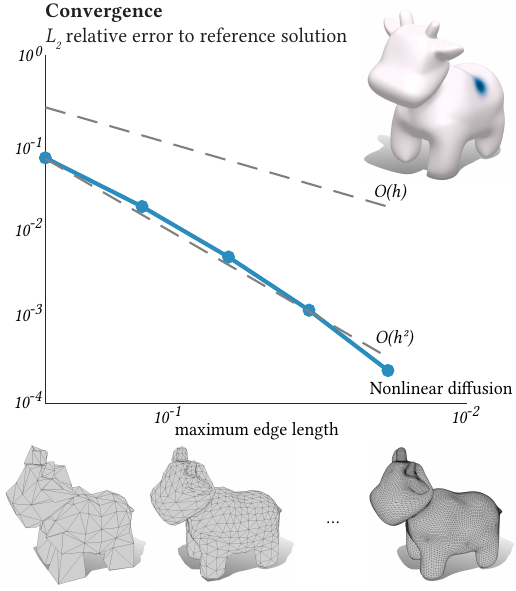}
  \caption{Convergence plot obtained for a single fixed time step $h=10^{-5}$ as the triangle mesh is refined, i.e., maximum edge length decreases ($x$-axis). The error ($y$-axis) is measured in the $L_2$ norm against the solution on the mesh of highest resolution.}
  \label{fig:space-convergence}
\end{figure}

\begin{table}
    \caption{\centering Determination of order for convergence in Fig. \ref{fig:space-convergence}}
    \vspace{-5pt}
    \centering
    \begin{tabular}{cc}
    \hline
    $\ell_{max}$ &  R  \\
    \hline
    \rowcolor{Gray}
    $2.56\times10^{-1}$& $1.864$  \\
    $1.22\times10^{-1}$& $1.932$ \\
    \rowcolor{Gray}
    $6.26\times10^{-2}$ & $2.036$ \\
    $3.37\times10^{-2}$ & $2.308$  \\
    \hline
    \end{tabular}
    \vspace{.2in}
    \label{tab:space-convergence}
\end{table}

\paragraph*{Nonlinear diffusion} In this case, we have quadratic constraints:

\begin{equation} \label{eq:quad-constraint}
\begin{array}{r@{\ }l}
    \argmin_{u^{(2)}}\; &\textstyle\sum_i {\big(u_n^{(2)}\big)}_i\\[2pt]
    \text{subject to } & u_n^{(2)} - u_n^{(1)} - h\;\omega_i \textstyle\sum_{j\sim i}\textstyle\sum_{k=1}^3 (G_ku_n^{(2)}\odot G_ku_n^{(2)})_j  \geq 0, \\[2pt]
    &\qquad\qquad\qquad\qquad\qquad\qquad\qquad i=1,\ldots,|V|.
\end{array}
\end{equation}

While the program above is convex, the constraints involve general convex quadratic forms.  It can be more efficient to convert such programs to second-order cone program (SOCP) standard form, an optimized form for standard solvers like SDPT$3$. Hence, we reformulate the above quadratic constraint as a second-order cone. In particular, algebraic manipulation shows that \eqref{eq:quad-constraint} is equivalent to the following constraints

\begin{equation}
    \left\|\begin{array}{c}\nicefrac{(1-z_j)}{2} \\ (Gu_n^{(2)})_j\end{array}\right\|_2 \leq \frac{(1+z_j)}{2},\; j=1,\ldots,|F|
\end{equation}

\begin{equation}
    \big(u_n^{(2)}\big)_i - \big(u_n^{(1)}\big)_i - h\;\omega_i\textstyle\sum_{j\sim i}z_j \geq 0,\; i=1,\ldots,|V|.
\end{equation}

\paragraph*{G-equation} In this case, we simply have constraints using norms, which is already in SOCP form.  Hence, we can use the conic solver efficiently without further reformulation. We implement our optimization as follows
\begin{equation}
\begin{array}{r@{\ }l}
    \argmin_{u^{(2)}}\; &\textstyle\sum_i {\big(u_n^{(2)}\big)}_i\\
    \text{subject to } &h\;\omega_i \textstyle\sum_{j\sim i}\big(\textstyle\sum_{k=1}^3  (\Phi_k \odot G_ku_n^{(2)})_j - \| (Gu_n^{(2)})_j\|_2\;\big)\\
    &\qquad\qquad\qquad +\; u_n^{(2)} - u_n^{(1)} \geq 0, \; i=1,\ldots,|V|.  
    \end{array}
\end{equation}

\paragraph*{Fokker-Planck equation} This equation is linear and requires no reformulation. Its implementation becomes:
\begin{equation}
\begin{array}{r@{\ }l}
    \argmin_{u^{(2)}}\; &\textstyle\sum_i {\big(u_n^{(2)}}\big)_i\\
    \text{subject to } &h\;\big((u_n^{(2)}\odot (G^\mathrm{T}M_F\Phi))_i+\omega_i \textstyle\sum_{j\sim i}\textstyle\sum_{k=1}^3 (\Phi_k \odot G_ku_n^{(2)})_j\big)\\
    &\qquad\qquad\qquad +\; u_n^{(2)} - u_n^{(1)} \geq 0, \; i=1,\ldots,|V|.
    \end{array}
\end{equation}

\begin{figure*}
  \centering
  \includegraphics[width=\linewidth]{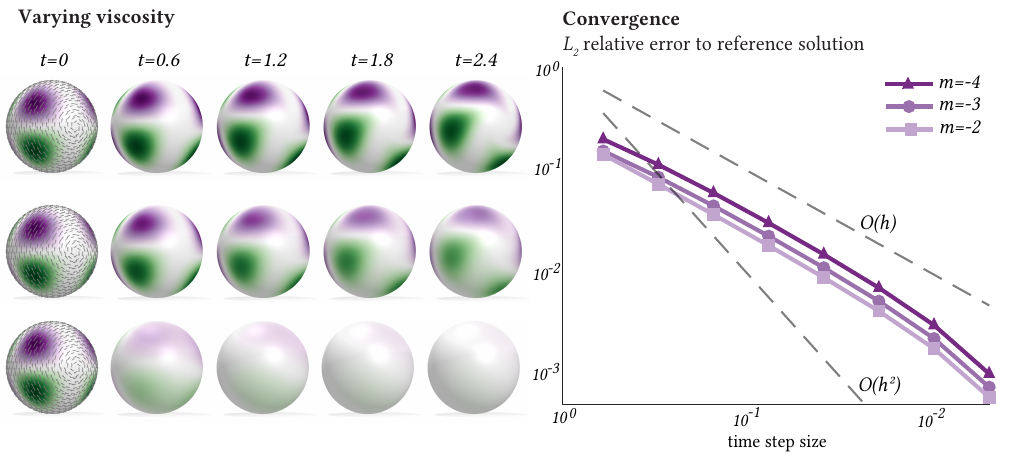}
  \caption{The effect of varying the viscosity parameter in the Fokker-Planck equation on a sphere with $81{,}920$ triangle faces and $40{,}962$ vertices.  The equation is shown evolving in time on the discrete sphere with viscosity $\varepsilon=10^{-4}$ (top row), $\varepsilon=10^{-3}$ (middle row), and $\varepsilon=10^{-2}$ (bottom row).}
  \label{fig:sphere-time-convergence}
\end{figure*}

\begin{figure}
  \centering
  \includegraphics[width=\linewidth]{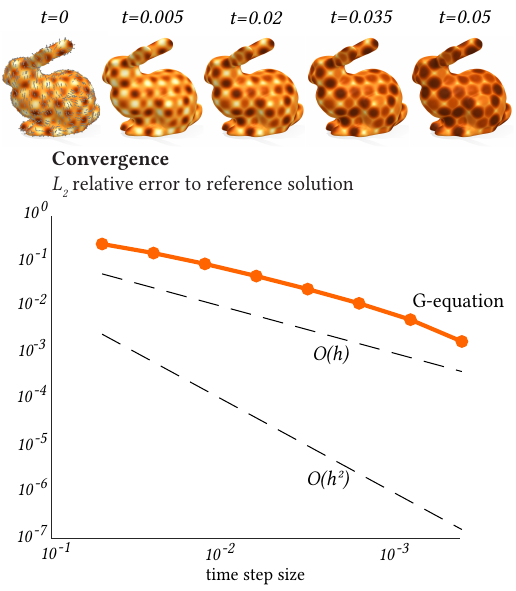}
  \caption{Convergence plot obtained for the time evolution of the $G$-equation (top row) as the time step size is decreased ($x$-axis). The error ($y$-axis) is computed using the $L_2$ norm against the solution obtained with smallest step size.}
  \label{fig:bunny-time-convergence}
\end{figure}

\begin{figure}
  \centering
  \includegraphics[width=\linewidth]{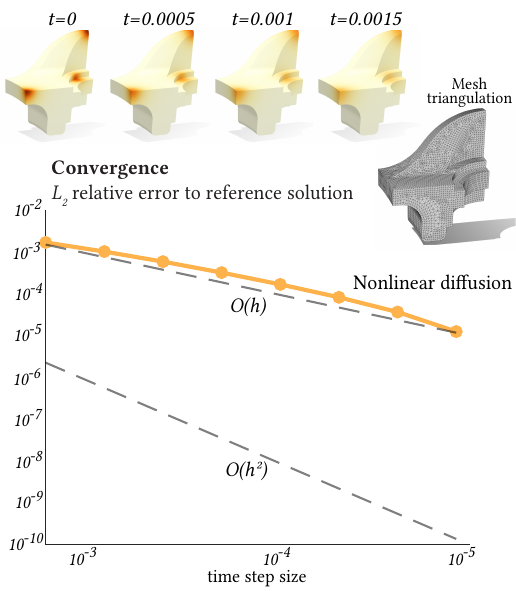}
  \caption{Convergence plot obtained for the time evolution of the nonlinear diffusion equation on a mesh with both varying triangulation quality and sharp features (top row) as the time step size is decreased ($x$-axis). The time evolution illustrations (top row) are show in exponential domain.}
  \label{fig:fandisk-time-convergence}
\end{figure}

\section{Numerical Experiments}

We demonstrate the versatility of our method for a variety of PDE on surfaces and validate its performance in this section. Our experiments were carried out in \textsc{Matlab} $2023\mathrm{a}$, on a macOS machine with $32$GB memory. All tests were run with solver tolerance $\delta^{\nicefrac{1}{2}}$, where $\delta=2.22\time 10^{-16}$ is the machine precision. Code with examples is included with this paper.

\subsection{PDEs on Surfaces}

Our main numerical contribution is a solver capable of handling a range of linear and nonlinear parabolic PDE. As will be shown later, this solver can be coupled with existing methods found in practical applications in computer graphics (see \S\ref{sec:applications}). For now, we focus on demonstrating the direct use of the solver to evolve PDE with various choices of parameters and meshes. Specifically,  given a velocity vector field $\Phi\in \R^{|F|\times3}$ defined per face on an  underlying computational mesh, we can use our framework to obtain the numerical time evolution of the $G$-equation equation and the Fokker-Planck equation as seen on the tori in Figure \ref{fig:tori-evolution}. We also refer the reader to Figures \ref{fig:sphere-time-convergence}, \ref{fig:bunny-time-convergence}, and \ref{fig:fandisk-time-convergence} for illustrations of time evolution of each of our example PDE. The order of convergence for these experiments is discussed in the next subsection.

\subsection{Convergence}

It is natural to ask whether the second-order parabolic PDE problems that we consider on curved surfaces admit closed-form solutions. To the best of our knowledge, such solutions are in general unavailable outside of a few simple cases.
Consider, e.g., the problem of solving the Fokker-Planck equation on the unit sphere. When the vector field $\Phi$ is zero, then the Fokker-Planck equation reduces to the heat equation, which is easy to solve using spherical harmonic expansion. However, when $\Phi$ is nonzero, the Hamiltonian depends on $\nabla u$, and it is typically quite difficult to express the first-order derivatives of a function in terms of a spherical harmonic expansion. We refer the reader to \S3 of \citet{barrera1985harmonics}, where the process of writing down such an expansion is described as “frustrating.”

To get close to measuring convergence to an analytical solution, we perform a self-convergence experiment. We use the solution for highest mesh resolution as the (nearly-)exact solution, denoted by $u^\ast$, and compare it to a sequence of solutions $u_N$ found for lower resolution meshes. We estimate the convergence order of the method by 
\begin{equation}\label{eq:order}
    R = \log_{2}{\left(\frac{u^\ast -u_N}{u^\ast-u_{\nicefrac{N}{2}}}\right)},
\end{equation}
where $N$ is either a space or time discretization size, i.e., maximum edge length size $\ell$ (for spatial convergence) or time step size $h$ (for time convergence).

In Figure \ref{fig:space-convergence}, we present an experiment indicating convergence under spatial refinement. While our choice of discretization is standard first-order FEM, we find that the order of convergence of our method under spatial refinement is second-order (see Table \ref{tab:space-convergence}). In Figure \ref{fig:bunny-time-convergence}, we show an experiment to determine convergence under time refinement for the evolutionary $G$-equation, and in Figure \ref{fig:sphere-time-convergence}, we do the same for the Fokker-Planck equation using various amounts of viscosity. The results are summarized in Table \ref{tab:time-convergence}. For both equations convergence is of at least first-order, and first-order convergence is consistent for various amounts of viscosity as demonstrated in the evolution of the Fokker-Planck. While Strang-Marchuk splitting encourages second-order convergence, we find empirically that the first-order nature of each step's method only guarantees first-order convergence for the scheme.

\begin{figure}[b]
  \centering
  \includegraphics[width=\linewidth]{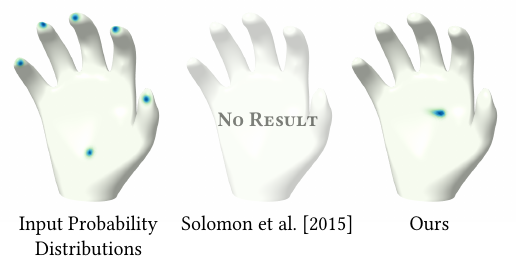}
  \caption{Barycenters with entropy value $\gamma=10^{-3}$ obtained using the convolutional method (middle) and our logarithmic diffusion method (right). The very sharp input probability distributions $\{\mu_i\}_{i=1}^6$ are shown together on the left.} 
  \label{fig:barycenters}
\end{figure}

\subsection{Robustness}

The performance of our method in meshes with varying triangulation and with sharp features is illustrated in Figure \ref{fig:fandisk-time-convergence}. In this case, our method requires smaller step sizes to achieve first-order convergence. 

\section{Applications} \label{sec:applications}

\subsection{Wasserstein Barycenter} \label{subsec:barycenter}

The ``log-sum-exp'' trick is a standard method used to stabilize numerical algorithms, including the Sinkhorn algorithm when using small amounts of entropy. This computation, however, is not possible using the method adapted to triangle meshes proposed in \cite{solomon2015convolutional}, which corresponds to taking the logarithm of a function undergoing tiny amounts of heat diffusion. We advocate for using our numerical framework and Proposition \ref{prop-heat} to solve \eqref{eq:logheat} to compute the result of heat diffusion on triangle meshes directly in the logarithmic domain, rather than diffusing in the linear domain and then taking the logarithm.

\begin{figure*}
  \centering
  \includegraphics[width=\linewidth]{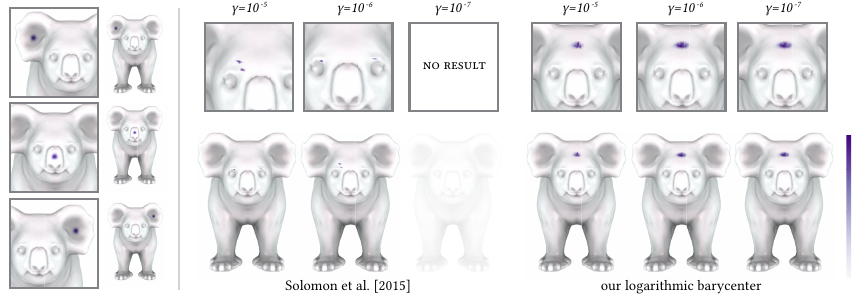}
  \caption{Small entropy coefficients can lead to numerical challenges using the state-of-the-art method (middle) for computing Wasserstein barycenters, which our logarithmic diffusion method (right) overcomes. Here, the Wasserstein barycenter of three input distributions (left) on a mesh with $99{,}994$ triangle faces and $49{,}999$ vertices is shown for different entropy coefficients.}
  \vspace*{10pt}
  \label{fig:barycenter-koala}
  \vspace{-0.1in}
\end{figure*}

In Figure \ref{fig:barycenters}, we demonstrate how our log-domain diffusion algorithm outperforms the state-of-the-art in computing Wasserstein barycenters on mesh surfaces. The method in \cite{solomon2015convolutional} fails to output a result for the entropy coefficient $\gamma=10^{-3}$ with very sharp probability distributions as input. Even excluding this failure case, the Wasserstein barycenters obtained via the convolutional method look qualitatively wrong for a number of entropy coefficients in Figure \ref{fig:barycenter-koala}, while our method remains numerically stable and obtains the expected visual results given the distributions.

\subsection{Measure Interpolation} \label{subsec:measure-interp}

Given a pair of initial and target distributions $\mu_0$ and $\mu_1$ on a triangular mesh, the same algorithm used to compute Wasserstein barycenters can be applied with weights $(1-t,t)$ to compute a time-varying sequence of distributions $\mu_t$ transitioning from initial to target, moving along geodesic paths. In Figure \ref{fig:dist-interp-horse}, we show a comparison to the convolutional method in \cite{solomon2015convolutional} for the distance interpolation task. We observe again that the convolutional method fails to obtain results with a small entropy coefficient for various pairs of weights, while our method succeeds in both smooth and noisy meshes, demonstrating the robustness of our method in this task.

\subsection{Numerical Integration for Fire and Flames} \label{subsec:numerical-fire}

\citet{osher1988curvature} provide a first-order scheme to approximate solutions to Hamilton-Jacobi equations, which is coupled with the Navier-Stokes equation by \citet{nguyen2002fire} in the physical simulation of fire for graphics applications. The first method is based on an upwind generalized first-order version of Godunov's scheme \cite{godunov1959method}. We compare our method for solving the $G$-equation to the method in \cite{osher1988curvature} and a Lax-Friedrichs scheme described in \cite{crandall1984hamiltonjacobi}, both standard methods of approximation for Hamilton-Jacobi equations, demonstrating how it could be effectively coupled in the same fashion for applications in the simulation of fire and flames.

The main drawback of the numerical method presented in \cite{osher1988curvature} is: (1) a CFL condition given the method's explicit nature, and (2) its limitation to regular grids. Although our method has longer runtime per time step, it offers more stability for larger range of time step sizes (see Fig. \ref{fig:cfl-violation}), and further, it can be implemented on \emph{curved} triangle mesh surfaces. We focus on the special case of flat domains for the comparison because the aforementioned methods are designed for that case.

In Figure \ref{fig:numerical-comaprison}, we show a comparison between our framework and the method in \cite{osher1988curvature} for the evolution of a front $u$ with propagation prescribed by $u_t-\|\nabla u\|_2=0$, that is, essentially the $G$-equation. Since this is the limiting case where $\varepsilon=0$, we add a small amount of viscosity $\mathcal{O}(10^{-6})$ to guarantee its solutions are continuous. This is a simple regularization technique for problems without continuous solutions first introduced by \citet{sethian1985fronts}. We apply our framework to the regularized equation $u_t-\|\nabla u\|_2=\varepsilon\Delta u$ and show numerically that the dissipation created by adding viscosity performs the same as the reference method in \cite{osher1988curvature}. We observe convergence to the same steady state and very small error at each time step of the front propagation. This comparison suggests that  our method could be used as more stable component in larger simulation pipelines.

\section{Discussion and Conclusion}

PDE appear everywhere in geometry processing, and second-order parabolic PDE are no exception. While there are many tools to handle simpler PDE, such as the heat equation, more general parabolic PDE are a challenge to standard time integration methods and discretization schemes. Our work establishes an effective time integration and spatio-discretization strategy to solve this class of PDE under mild assumptions on triangle mesh surfaces. In addition to substantial theoretical work, we showed several numerical experiments indicating the validity of our method. We have also demonstrated how our method can be used as a numerical solver component in graphics applications. In particular, we showed how our method overcomes limitations in optimal transport tasks over geometric domains and in numerical integration schemes used for larger simulation pipelines.

\begin{figure*}
  \centering
  \includegraphics[width=\linewidth]{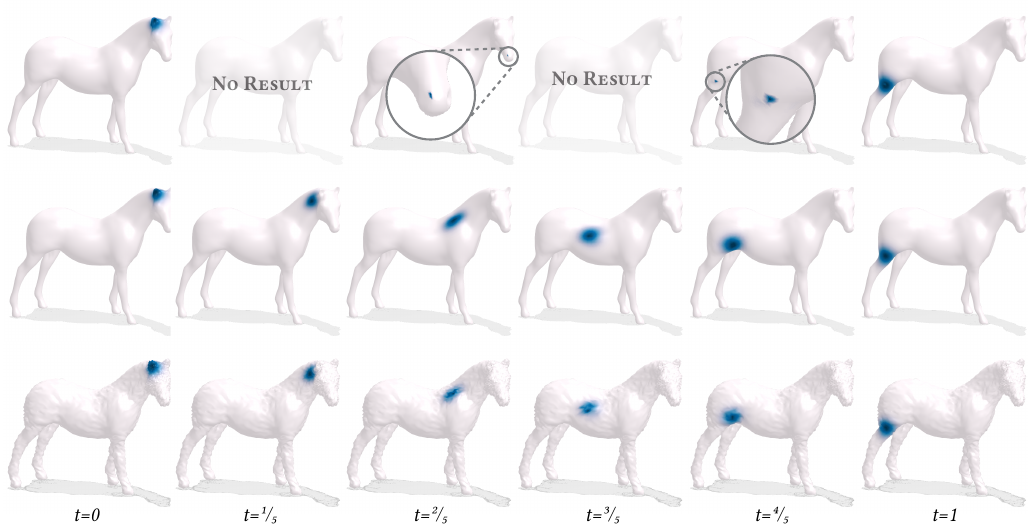}
  \caption{Wasserstein distance interpolation on a mesh with $76{,}736$ triangle faces and $38{,}370$ vertices. Top row: using the convolutional method proposed by \citet{solomon2015convolutional}. Middle row: using our logarithmic diffusion method. Bottow row: using our logarithmic diffusion method on a mesh with added uniformly random noise along normal directions. As seen on the top row, the convolutional method fails to obtain results for a small amount of entropy $\gamma=10^{-5}$.}
  \label{fig:dist-interp-horse}
\end{figure*}

While we have explored several practical implications of our framework, a number of interesting avenues for future work remain. In particular, for the Fokker-Planck equation, further experimental work includes extending results to other versions of this equation (e.g., nonlinear) and using the solutions obtained via our method together with the relationship to \eqref{eq:sde} to simulate Brownian motion on triangular surface meshes. Our work provides a meaningful first-step in this direction. On the practical side, an immediate avenue for future work is to derive an optimization algorithm to solve the problem in \eqref{eq:optprob} along the lines of \citet{edelstein2023convex} and compare it to our current off-the-shelf solution using CVX. On the theoretical front, we would like to explore an extension of our results to higher-order parabolic equations with Laplacian terms, such as the \textit{Kuramoto–Sivashinsky} equation.  Within the realm of second-order parabolic PDE of the form \eqref{eq:pdeform}, we are hopeful that our framework would provide a means to solve systems of coupled equations.

\begin{figure}
  \centering
  \includegraphics[width=\linewidth]{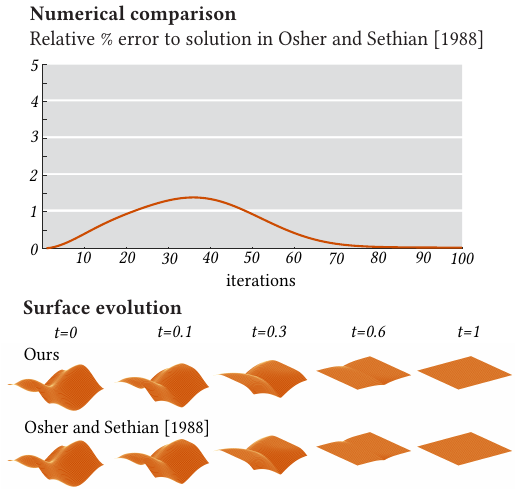}
  \caption{Surface moving with speed $\|\nabla u\|_2$. Here, the numerical comparison is done on a regular $50\times50$ regular grid using centered differences and periodic conditions; the time step size is $h=0.01$.}
  \label{fig:numerical-comaprison}
\end{figure}

\begin{acks}

The authors would like to thank Nestor Guillen for his thoughtful insights and feedback, as well as Paul Zhang and Yu Wang for proofreading.

Leticia Mattos Da Silva acknowledges the generous support of the Schwarzman College of Computing Fellowship funded by Google Inc. and the MathWorks Fellowship. Oded Stein was supported by the Swiss National Science Foundation’s Early Postdoc.Mobility Fellowship. Justin Solomon acknowledges the generous support of Army Research Office grants W911NF2010168 and W911NF2110293, of Air Force Office of Scientific Research award FA9550-19-1-031, of National Science Foundation grant CHS-1955697, from the CSAIL Systems that Learn program, from the MIT–IBM Watson AI Laboratory, from the Toyota–CSAIL Joint Research Center, from a gift from Adobe Systems, and from a Google Research Scholar award.

\end{acks}

\begin{figure}
  \centering
  \includegraphics[width=\linewidth]{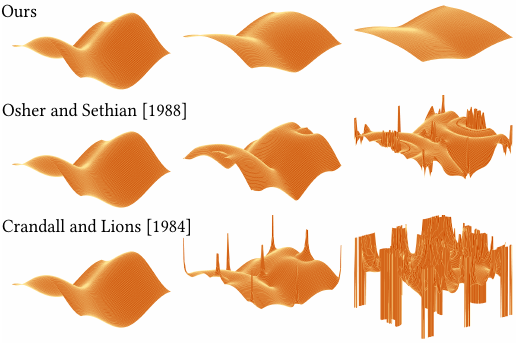}
  \caption{Standard methods used in larger graphics pipelines for simulation of fire and flames become numerically unstable for time step and spatial discretization size outside a certain CFL condition, while our method remains stable. Here, the same comparison in Fig. \ref{fig:numerical-comaprison} is done, now on a regular $100\times 100$ grid and time step size $h=0.1$.}
  \label{fig:cfl-violation}
  \vspace{-0.2in}
\end{figure}

\bibliographystyle{ACM-Reference-Format}
\bibliography{parabolic-pdes}


\appendix

\newpage
 
\section{Comparison Principle (Theorem \ref{thm:comparison})} \label{sec:comparison}

In this section, we provide a sketch of the proof for the comparison principle. The formal proof can be found in  \S5 of \cite{barles2013viscosity}, Theorem $5.2$. 

First, we note that by adding an appropriate multiple of $t$ to $w$, we may assume that $w$ is a strict subsolution. In particular, there exists a constant $\eta>0$ such that
\begin{equation}
    \frac{\partial w}{\partial t} + H(x,q,w) \leq -\eta<0. \label{eq:sub}
\end{equation}
Since $u$ is a supersolution, we also have
\begin{equation}
    \frac{\partial u}{\partial t} + H(x,q,u) \geq 0. \label{eq:sup}
\end{equation}
The key idea in this proof is to obtain a contradiction by assuming that the maximum value of $w-u$ is positive, i.e., we suppose $M=\max\{w-u\}> 0$. Now consider the test function 
\begin{equation*}
    \psi_{\beta,\alpha}(x,t,y,s) = w(x,t) - u(y,s) - \frac{1}{2\beta^2}\|x-y\|_2^2  - \frac{1}{2\alpha^2}\|t-s\|_2^2,
\end{equation*}
where $\beta,\alpha$ are some small positive constants. 

Suppose $(\bar{x},\bar{t},\bar{y},\bar{s})$ is a maximum argument of $\psi_{\beta,\alpha}$ and denote its maximum value by $\bar{M}$. Under the assumption that $\bar{M}>0$, we will apply Lemma $5.2$ in \cite{barles2013viscosity}, which says that in the infinitesimal limit we have $(\bar{x},\bar{t})\to(\bar{y},\bar{s})$ and $\bar{M}\to M$. In other words, for sufficiently small $\beta,\alpha$, the test function is $\psi_{\beta,\alpha}$ is a close approximation of $w-u$. By writing \eqref{eq:sub}--\eqref{eq:sup} in terms of the test function $\psi_{\beta,\alpha}$ and combining the inequalities, one can deduce that for sufficiently small $\beta,\alpha$, we have
\begin{align*}
    -\eta&\geq H\left(\bar{x},\frac{\bar{x} - \bar{y}}{\beta^2},w(\bar{x},\bar{t})\right) - H\left(\bar{y},\frac{\bar{x} - \bar{y}}{\beta^2},u(\bar{y},\bar{s})\right)\\[2pt]
    &\geq H\left(\bar{x},\frac{\bar{x} - \bar{y}}{\beta^2},w(\bar{x},\bar{t})\right) - H\left(\bar{x},\frac{\bar{x} - \bar{y}}{\beta^2},u(\bar{y},\bar{s})\right)\\
    &\quad\quad- L\left(1+ \left\|\frac{\bar{x} - \bar{y}}{\beta^2}\right\|_2\right)\|\bar{x} - \bar{y}\|_2\\[2pt]
    &\geq c(w(\bar{x},\bar{t}) - u(\bar{y},\bar{s})) - L\left(1+ \left\|\frac{\bar{x} - \bar{y}}{\beta^2}\right\|_2\right)\|\bar{x} - \bar{y}\|_2
\end{align*}
where we have used both assumptions (A2) and (A3) on $H$. Taking $\beta,\alpha\to0$, the above implies $0> c(w(\bar{y},\bar{s}) - u(\bar{y},\bar{s}))$, a contradiction.

\section{Existence and Uniqueness (Theorem \ref{thm:uniqueness})} \label{sec:uniqueness}

We now present a sketch of the proof for the existence of a unique viscosity solution to $\nicefrac{\partial u}{\partial t}+H(x,q,u)=0$. The formal proof can be found in \S7 of \cite{barles2013viscosity}, Theorem $7.1$. 

Let $\mathcal{S}$ denote the set of all subsolutions, and let $u$ be the pointwise supremum of $\mathcal{S}.$ Then $u$ need not be continuous, or even semi-continuous, but we can take its lower-semicontinuous and upper-semicontinous envelopes $u_\ast$ and $u^\ast$, respectively, which satisfy the inequality $u_\ast\leq u \leq u^\ast$. The main idea of this proof is is to establish the reverse inequality $u^\ast\leq u_\ast$, which would imply $u^\ast= u = u_\ast$. To do this it suffices to prove the following:
\begin{lemma} \label{upperenv}
    The upper envelope $u^*$ is a viscosity subsolution. 
\end{lemma}
\begin{lemma}  \label{lowerenv}
    The lower envelope $u_*$ is a viscosity supersolution. 
\end{lemma}
To obtain Lemma \ref{upperenv}, the key insight is that the supremum of a set of subsolutions is itself a subsolution. This can be shown in three steps: first, one can show it for a pair of subsolutions, second for a countable set of subsolutions, and, finally, for an arbitrary set of subsolutions by reduction to the countable case.

As for Lemma \ref{lowerenv}, we obtain it by contradiction. Suppose $u_*$ is not a viscosity supersolution, then without loss of generality there exists a function $\phi$ such that $u_*-\phi$ is minimized with value zero at a point $(x,t)=(x_0,t_0)$ such that $\nicefrac{\partial \phi(x_0,t_0)}{\partial t}+H(x_0,\nabla \phi(x_0,t_0),u(x_0,t_0))=0$. One can then verify that for $\zeta>0$ small enough, the function $\max\{ u(x,t), \phi(x,t) + \zeta$ $-\|x-x_0\|_2^4-\|t-t_0\|_2^4\}$ is itself a subsolution that is greater than $u$ on a neighborhood of $(x_0,t_0)$. This contradicts the definition of $u$ as the pointwise supremum of $\mathcal{S}$. Lipschitz continuity is used to ensure that this argument extends to the time boundary $t=0$. It then follows from the comparison principle that $u^*\leq u_*$ as they share the same initial conditions.

\end{document}